\documentclass[11pt]{amsart}
\usepackage{amsmath, amssymb}
\usepackage{amsfonts}
\usepackage{mathrsfs}

\usepackage{amsthm}

\usepackage{color}
\usepackage{amsfonts}
\usepackage{amscd}
\usepackage{comment}
\usepackage{epsfig}

\usepackage[colorlinks=true, pdfstartview=FitV, linkcolor=red, citecolor=blue, urlcolor=green]{hyperref}
\newcommand{\arxiv}[1]{\href{http://arxiv.org/abs/#1}{\tt arXiv:\nolinkurl{#1}}}

\allowdisplaybreaks

\newcommand{\cX}{\mathcal{X}}

\newtheorem{theorem}{Theorem}[section]
\newtheorem{lemma}[theorem]{Lemma}
\newtheorem{proposition}[theorem]{Proposition}

\newtheorem*{theorem*}{Theorem}
\theoremstyle{remark}
\newtheorem{remark}[theorem]{Remark}
\newtheorem*{rem*}{Remark}

\newtheorem{claim}[theorem]{Claim}

\numberwithin{equation}{section}

\newcommand{\cz}{Calder\'{o}n--Zygmund\ }
\newcommand{\cF}{\mathcal{F}}
\newcommand{\cL}{\mathcal{L}}
\newcommand{\cG}{\mathcal{G}}
\newcommand{\cE}{\mathcal{E}}
\newcommand{\cA}{\mathcal{A}}
\newcommand{\cB}{\mathcal{B}}
\newcommand{\cD}{\mathcal{D}}

\newcommand{\1}{\mathbf{1}}
\newcommand{\fS}{\mathfrak{S}}

\newcommand{\bl}{\mathbf{1}}

\newcommand{\La}{\langle}
\newcommand{\Ra}{\rangle}

\newcommand{\gb}{\mathbf{\dot g}}

\newcommand{\wt}{\widetilde}

\newcommand{\Z}{\mathbb{Z}}
\newcommand{\N}{\mathbb{N}}
\newcommand{\R}{\mathbb{R}}

\newcommand{\E}{\mathbb{E}}

\newcommand{\child}{\operatorname{ch}}
\newcommand{\ch}{\operatorname{ch}}

\newcommand{\ci}[1]{_{ {}_{\scriptstyle #1}}}

\newcommand{\ut}[1]{^{\scriptstyle \text{\rm #1}}}
\renewcommand{\labelenumi}{(\roman{enumi})}
\newcounter{vremennyj}
\newcommand\cond[1]{\setcounter{vremennyj}{\theenumi}\setcounter{enumi}{#1}\labelenumi\setcounter{enumi}{\thevremennyj}}

\newcommand{\rk}{\operatorname{rk}}
         
\begin{document}
\title[Two weight \texorpdfstring{$L^p$}{L**p} estimates for paraproducts]{Two weight \texorpdfstring{$L^p$}{L**p} estimates for paraproducts in non-homogeneous settings}

\author{Jingguo Lai}

\address{Department of Mathematics \\ Brown University \\ Providence, RI 02912 \\ USA}
\email{jglai@math.brown.edu}
\author{Sergei Treil}
 \thanks{Supported  in part by the National Science Foundation under the grant DMS-1301579.}
\address{Department of Mathematics \\ Brown University \\ Providence, RI 02912 \\ USA}
\email{treil@math.brown.edu}




\begin{abstract}
We give a necessary and sufficient condition for the  two weight $L^p$-estimates for paraproducts in non-homogeneo\-us settings, $1<p<\infty$. We are mainly interested in the case $p\ne 2$, since the case $p=2$ is a well-known and easy corollary of the Carleson embedding theorem. The necessary and sufficient condition  is given in terms of testing conditions of Sawyer type:  for $p\le 2$ only one (``direct'') testing condition is required, but for $p>2$ both ``direct'' and ``adjoint'' testing conditions are needed. An interesting feature is that  the ``adjoint'' testing condition is that it is a testing condition not for the adjoint of the paraproduct, but for some auxiliary operator. 
\end{abstract}

\maketitle

\setcounter{tocdepth}{1}
\tableofcontents

\setcounter{tocdepth}{2}

\section{Introduction and main results}

The paper is devoted to the two weight estimates of paraproducts in the non-homogeneous martingale situation. Paraproducts play an important role in the investigation of the weighted inequalities for the singular integral operators. The $L^2$-boundedness of paraproducts is easy, a necessary and sufficient condition follows immediately from the Carleson embedding theorem. 

This necessary and sufficient condition can be stated as a \emph{testing condition}, i.e.~a dyadic paraproduct is bounded in $L^2$ if and only if there is a uniform estimate on all dyadic intervals (cubes). In the classical homogeneous (non-weighted) situation the $L^2$ boundedness (or $L^{p_0}$ boundedness, for some $1<p_0<\infty$) is equivalent to the boundedness of the paraproduct in \emph{all} $L^p$, $1<p<\infty$. 

The non-homogeneous situation is much more interesting, even in the non-weighted case. Namely, it was shown in \cite{T2} that in this case the testing condition is still is  necessary and sufficient, but it now depends on $p$: the boundedness in $L^{p_0}$ implies the boundedness in $L^p$ with $1<p\le p_0$ but not in $L^p$ with $p_0<p<\infty$.

Two weight case becomes even more interesting: while it is not hard to show that for $p\le 2$ the testing condition is still sufficient for the boundedness, we will present a counterexample showing that for $p>2$ the testing condition alone does not work. 

It will be shown in the paper that adding a second, ``dual'' testing condition we get necessary and sufficient conditions for the two weigh estimates for $p>2$. 

It has been understood for a long time, that in many cases one needs two testing conditions, one for the operator, the other one for the adjoint, to get a sufficient condition for the boundedness. That was the case, for example for many positive operators (i.e.~integral operators with positive kernels), in particular for the so-called \emph{positive dyadic operators}. That was also the case for  \cz operators, but in this case the bounds on paraproduct followed from one testing condition, and the dual testing condition implied bounds on the adjoint paraproduct. 

But in our case we need 2 testing condition for the estimates of paraproduct! Also, our ``dual'' testing condition is not a testing condition for the adjoint paraproduct, but for some auxiliary operator.

\subsection{Preliminaries}

The theory of $L^2$ two weight estimates is now well developed for many classes of singular operators, including the so-called ``well localized'' operators, treated in \cite{NTV2} where necessary and sufficient condition of Sawyer type was proves for such operators. 

The Sawyer type conditions are essentially the testing conditions on the operator and its adjoint, meaning that to prove that the operator is bounded, it is sufficient to check that the operator and its adjoint are (uniformly) bounded on characteristic functions $\1\ci Q$ of cubes $Q$. We say ``essentially'' here, because in many cases this testing condition can be weakened a bit, like restricting the result only to the cube $Q$, or in case of martingale operators also considering   smaller sums.

The recent achievements in the theory of $L^2$ two weight estimates include sufficiency of Sawyer type estimates for Hilbert Transform \cite{LSSU}, \cite{L}, for the Cauchy Transform \cite{LSSUW}, for the Riesz transforms \cite{LW}. 

As for the $L^p$ theory for $p\ne 2$ everything is well understood for the ``positive'' operators, including the maximal function \cite{Sawyer_2weight-max_1982}, fractional and Poisson integrals \cite{Sawyer_2weightFract-1988}, the so-called positive dyadic operators \cite{LSU}, \cite{T1}. 

As for the more ``singular'' operators we should first mention \cite{S}, where a necessary and sufficient condition for the two weight estimates of the vector-valued martingale operators was obtained. Such operators can be interpreted as a simplified ``model'' for paraproducts, and the necessary and sufficient conditions were ``kind of'' Sawyer type conditions. More precisely, the ``direct'' condition was exactly the Sawyer condition for the operator, but the adjoint condition was a more restrictive one, in a sense that for each interval $I$ one has to check the boundendness not on one function, but on an infinite family of (vector-valued) functions. 

Very interesting recent papers \cite{Vu1}, \cite{Vu2} give necessary and sufficient conditions for the two weight estimates of the well localized operators and of the Haar shifts, but again the conditions were more restrictive than the Sawyer type ones: again, for each interval $I$  one has to check the roundedness on infinitely many vector-valued functions. 

In this paper we consider paraproducts and the corresponding vector-valued martingale operators and get necessary and sufficient condition for the $L^p$ two weight estimates. Our conditions are exactly Sawyer type conditions, meaning that for each interval we have to test only one function $\1\ci I$. However, while the ``direct'' condition is exactly the testing condition for the original operator, the ``adjoint'' one tests an auxiliary operator, acting not between duals of the original  $L^p$ spaces, but between the duals of $L^{p/2}$ spaces.  We should mention that for $p\le2$  the ```direct'' testing conditions alone is sufficient; for $p>2$ both ``direct'' and ``adjoint'' conditions are required.

\subsection{The setup}

The general setup is a s follows. We consider a $\sigma$-finite measure space $(\cX, \fS, \mu)$ with a filtration (i.e.~with increasing sequences of $\sigma$-algebras) $\fS_n\subset \fS_{n+1}\subset \fS$, $n\in\Z$.

We assume that each $\sigma$-algebra $\fS_n$ is \emph{atomic}, i.e.~that there exist a countable  collection $\cL_n$ of disjoint sets of finite measure (atoms) such that each $A\in\fS_n$ is a union of sets $I\in\cL_n$. 

The fact that $\cL_n\subset \cL_{n+1}$ means that each $I\in\cL_{n}$ is at most countable union of $I'\in\cL_{n+1}$. 

We denote by $\cL=\bigcup_{n\in\Z} \cL_n$ the collection of all atoms (in all generations). 
We will also use the word \emph{intervals} for the atoms $I\in\cL$.  

The standard example will be the filtration given by a dyadic lattice in $\R^d$ but  with the underlying measure being an arbitrary Radon measure $\sigma$. 

We will allow a situation  when an atom $I$ belongs to several (even infinitely many) generations $\cL_n$. 
However, we will not allow $I$ to be in all generations, because in this case nothing interesting happens on the interval $I$. 

There are two (equivalent) ways to define paraproduct.  One is a ``probabilistic'' approach, when one uses discrete time $n$ (index in $\fS_n$) and considers standard in probability martingale differences; in this approach by an atom $I\in \cL_n$ we understand the pair $(I,n)$. The other one is a ``geometric'' approach, where one does not use the time $n$, and treats atoms $I$ just as sets. 

Let us start with the ``probabilistic'' approach. 
For $I\in\cL_n$ define its \emph{martingale} or \emph{time} children by
\[
\ch\ut t (I,n) := \{ I'\in \cL_{n+1}: I'\subset I\}
\]

For an interval $I\in\cL$ and a measure $\nu$ such that $\nu(I)<\infty$ define the averaging operator $\E\ci I^\nu$
\begin{align} \label{Av-01}
\mathbb{E}\ci{I}^\nu f=\langle f\rangle\ci{I,\nu}\mathbf{1}\ci{I}=\left(\nu(I)^{-1}\int_If d\nu\right)\mathbf{1}\ci{I};
\end{align}
we use the notation $\La f\Ra\ci{I,\nu}$ for the average
\[
\La f\Ra\ci{I,\nu} :=\nu(I)^{-1} \int_I fd\nu
\] 
with the understanding that $\La f\Ra\ci{I,\nu} =0$ if $\nu(I)=0$.

We  define the martingale difference $\Delta\ci{(I,n)}^\nu$ as 
\[
\Delta\ci{(I, n)}^\nu := \sum_{I'\in \ch(I,n)} \E\ci{I'}^\nu \quad -\quad \E\ci I^\nu
\]

Let $\mu$, $\nu$ be  locally finite measures (i.e.~finite on all $I\in\cL$) and let $b$ be $\nu$-integrable on each $I\in\cL$. 
Define the two weight  paraproduct $\pi=\pi_b = {\pi_{b}^{\mu, \nu}}$ by 
\begin{align}\label{para-01}
\pi_{b}^{\mu,\nu} f=\sum_{n\in\Z}\sum_{I\in\mathcal{L}_n}\left( \mathbb{E}\ci{I}^\mu f\right)\left(\Delta\ci{(I,n)}^\nu b\right).
\end{align}

Note that if an interval $I$ belongs to several $\cL_n$, all the term in the paraproduct corresponding to this interval, with a probable exception of the term with the largest $n$ are $0$. So we can write the paraproduct differently, not writing zero terms.

Namely, one can define the rank $\rk I$ of $I\in L$ as supremum of $n$ such that $I\in\cL_n$ (it can be equal $+\infty$). Then for $I\in\cL$, $\rk(I)=n$ define the children $\ch(I)$ to be the collection of atoms $I'\in\cL_{n+1}$ such that $I'\subset I$; if $\rk(I)=\infty$ we set $\ch(I):=\{I\}$. 

Then we can define the martingale difference $\Delta\ci I$ as 
\[
\Delta\ci{I}^\nu := \sum_{I'\in \ch(I)} \E\ci{I'}^\nu \quad -\quad \E\ci I^\nu
\]
and rewrite the paraproduct $\pi=\pi_b = {\pi_{b}^{\mu, \nu}}$ as 
\begin{align}\label{para-02}
\pi_{b}^{\mu,\nu} f=\sum_{n\in\Z}\sum_{I\in\mathcal{L}_n}\left( \mathbb{E}\ci{I}^\mu f\right)\left(\Delta\ci{I}^\nu b\right).
\end{align}
This is what we call the \emph{geometric} notation for paraproducts. 

\begin{remark}
All results and proofs in the paper are valid in both ``geometric'' and ``probabilistic'' settings. For the sake of brevity of writing we will use the ``geometric'' notation. However, we should emphasize that everything can be easily (formally) translated to the ``probabilistic'' setup: one just has to understand an atom $I\in\cL_n$ as a pair $(I, n)$, and inclusion $I\subset J$ for atoms $I\in\cL_n$ and $J\in \cL_k$ should be understood as inclusion for sets plus the inequality $n\ge k$.  The union or intersection of atoms in this case should be understood as union/intersection of the corresponding sets and the ``time'' $n$ should not be taken into account. 
\end{remark}

\begin{remark}
When we introduced paraproduct we assumed that we are given a locally integrable function $b$ (symbol of the paraproduct) whose martingale differences $\Delta^\nu\ci I b$ are involved in the definition. However, that is not necessary,  the results of the paper hold if we just assume that we are given a sequence of martingale differences $b\ci I =\Delta\ci I^\nu b$, without assuming that they correspond to a function $b$. 

We only assume that each martingale difference $b\ci I$ is an integrable function supported on $I$, constant on children of $I$ and $\nu$-orthogonal to constants, i.e.~that $\int_I b\ci I d\nu=0$. 

In this case we will still use $b$ for the symbol of the paraproduct, but we understand $b$ as the collection $b=\{b\ci I\}\ci{I\in\cL}$ of all martingale differences. But we will still use the symbol $\Delta\ci I^\nu b$ for $b\ci I$ to remind us about the origins. 
\end{remark}

\subsection{Main results}
In this paper we investigate when  $\pi_{b}^{\mu,\nu}$ is a bounded map from $L^p(\mu)$ to $L^p(\nu)$ for $1<p<\infty$. 

\subsubsection{Estimates for paraproducts}For $I\in \cL$ denote $\cL(I) :=\{I'\in\cL:I'\subset I\}$. 

\begin{theorem}
\label{thm1}
Let $1<p\le 2$. Then the estimate 
\begin{align}\label{eq4}
\|\pi_b^{\mu,\nu} f\|\ci{L^p(\nu)}^p\le A^p \|f\|\ci{L^p(\mu)}^p,
\end{align}
holds if and only if for all $J\in \cL$
\begin{align}\label{eq5}
\int_J\Biggl|\sum_{I\in\mathcal{L}(J)}\Delta\ci{I}^\nu b\Biggl|^{p}d\nu\le B^p\mu(J).
\end{align}
Moreover, for the best constants we have $c(p)B\le A\le C(p) B$, where the constants $c(p)$, $C(p) $ depends only on $p$. 
\end{theorem}

\begin{theorem}
\label{t:para-p>2}
Let $2<p<\infty$. The estimate \eqref{eq4} 
holds if and only if \eqref{eq5} and 
\begin{align}
\label{eq6}
\int_{J}\left[\sum_{I\in\mathcal{L}(J)}
\frac{\nu{(I)}}{\mu{(I)}}\mathbb{E}\ci{I}^\nu\left( |\Delta\ci{I}^\nu b |^2\right)\right]^{r'}d\mu\le B_*^{r'} \nu(J), \qquad r=p/2
\end{align}
hold for all $J\in\cL$. 

Moreover, for the best constants we have  $c(p) \max\{B, B_* \}\le A \le C(p) \max\{B,B_*\}$. 
\end{theorem}

\begin{remark}
The one weight case $\nu=\mu$ was treated in \cite{T2}, where it was shown that \eqref{eq4} holds if and only if \eqref{eq5} holds for all $1<p<\infty$. However, in the two weight case, as we will see below in Section \ref{s:cEx}, the condition \eqref{eq5} alone is not sufficient if  $2<p<\infty$. 
\end{remark}

\subsubsection{Littlewood--Paley inequalities and estimates of generalized vector paraproducts}

For a function $f\in L^p(\nu)$, $1\le p \le\infty$ define its \emph{square function}
\begin{align}
\label{SqFn}
S^\nu f := \left( \sum_{I\in\cL} |\Delta\ci I^\nu f|^2 \right)^{1/2}. 
\end{align}

We will need the following well-known result. 
\begin{theorem}[Littlewood--Paley estimates]
\label{thm2}
For $f\in L^p(\nu)$, $1<p<\infty$
\[
\| S^\nu f \|\ci{L^p(\nu)} \le C(p) \|f\|\ci{L^p(\nu)} . 
\]
Moreover, if a function $f\in L^p(\mu)$ is decomposed $f=\sum_{I\in\mathcal{L}}\Delta\ci{I}^\nu f$, then 
\[
c(p) \|f\|\ci{L^p(\nu)}  \le \| S^\nu f \|\ci{L^p(\nu)} . 
\]
Here the constants $0<c(p)<C(p)<\infty$ depend only on $p$. 
\end{theorem}



This theorem allows us to reduce Theorems \ref{thm1} and \ref{t:para-p>2} to estimates of generalized vector-valued paraproducts. 

Following \cite{T2} define the Triebel--Lizorkin type space $\gb_p^q(\cL, \nu)$ as the space of sequences 
$s=\{s\ci I\}_{I\in\cL}$ such that
\[
\| s\|_{\gb_p^q(\cL)} := \biggl\| \biggl( \sum_{I\in\cL} |s\ci I |^q  \1\ci I \biggr)^{1/q} \biggr\|_{L^p}  <\infty  .
\]

For $I\in \cL$ let $\widehat I$ denote the parent of $I$. Let $\beta=\{\beta\ci{I,I'}\}\ci{I\in\cL, I'\in\ch(I)}$ be a numerical sequence. Define the \emph{generalized vector paraproduct}  $\Pi_\beta$ acting from $L^p(\mu)$ to the space of sequences, 
\[
\Pi_\beta f = \left\{ \La f\Ra\ci{\widehat I,\mu} \beta\ci{\widehat I, I} \right\}\ci{I\in\cL}
\]

\begin{remark}
\label{r:beta}
By the Littlewood--Paley estimates (Theorem \ref{thm2}) we can see that the estimate \eqref{eq4} is equivalent (up to a constant depending only on $p$) to the bounds on the operator $\Pi_\beta : L^p(\mu)\to \gb_p^2(\cL, \nu)$, where $\beta\ci{\widehat I, I} $ is the value of $\Delta\ci{\widehat I}^\nu$ on $I$.
\end{remark}

Theorem below hold for arbitrary operators $\Pi_\beta$, not only of the ones that came from the paraproducts $\pi_b^{\mu,\nu}$. For $I\in \cL$ define  
\begin{align}
\label{b_I}
b\ci I:= \sum_{I'\in\ch(I) } \beta\ci{I,I'} \1\ci{I'}.
\end{align}

\begin{theorem}
\label{t:SqPara-p<2}
Let $1<p\le 2$. Then the estimate 
\begin{align}
\label{SqEst-01}
\|\Pi_\beta f\|\ci{\gb_p^2(\nu)}^p\le A^p \|f\|\ci{L^p(\mu)}^p  \qquad \forall f\in L^p(\mu),
\end{align}
holds if and only if for all $J\in \cL$
\begin{align}
\label{SqTest-01}
\int_J\Biggl(\sum_{I\in\mathcal{L}(J)}\bigl|b\ci I \bigr|^2\Biggl)^{p/2}d\nu\le B^p\mu(J).
\end{align}
Moreover, for the best constants we have $B\le A\le C(p) B$, where the constant $C(p) $ depends only on $p$. 
\end{theorem}

\begin{theorem}
\label{t:SqPara-p>2}
Let $2<p<\infty$. The estimate \eqref{SqEst-01} 
holds if and only if \eqref{SqTest-01} and 
\begin{align}
\label{TestAdjPi}
\int_{J}\left[\sum_{I\in\mathcal{L}(J)}
\frac{\nu{(I)}}{\mu{(I)}}\mathbb{E}\ci{I}^\nu\left( | b\ci I |^2\right)\right]^{r'}d\mu\le B_*^{r'} \nu(J), \qquad r=p/2
\end{align}
hold for all $J\in\cL$. 

Moreover, for the best constants we have $ \max\{B, c(p) B_*^{1/2} \}\le A \le C(p) \max\{B,B_*^{1/2}\}$. 
\end{theorem}

Theorems \ref{t:SqPara-p<2} and \ref{t:SqPara-p>2} are particular cases of the theorems below. 

\begin{theorem}
\label{t:SqPara-p<q}
Let $1<q<\infty$, $1<p\le q$. Then the estimate 
\begin{align}
\label{SqEst-01-q}
\|\Pi_\beta f\|\ci{\gb_p^q(\nu)}^p\le A^p \|f\|\ci{L^p(\mu)}^p  \qquad \forall f\in L^p(\mu),
\end{align}
holds if and only if for all $J\in \cL$
\begin{align}
\label{SqTest-01-q}
\int_J\Biggl(\sum_{I\in\mathcal{L}(J)}\bigl|b\ci I \bigr|^q\Biggl)^{p/q}d\nu\le B^p\mu(J).
\end{align}
Moreover, for the best constants we have $B\le A\le C(p,q) B$, where the constant $C(p,q) $ depends only on $p$ and $q$. 
\end{theorem}

\begin{remark}
\label{r:Cu}
The above Theorem \ref{t:SqPara-p<q}  was proved in \cite{Cu} even in higher generality: the functions $b\ci I\ge0$ there could be arbitrary measurable functions supported on $I$, and the case $q=\infty$ was also treated there. 

For the convenience of the reader we present a proof here: the proof essentially the same as the proof in \cite{Cu}, although in \cite{Cu} due to a clever choice of the stopping parameter a better constant is obtained. 
\end{remark}

\begin{theorem}
\label{t:SqPara-p>q}
Let $1<q<\infty$, $q<p<\infty$. The estimate \eqref{SqEst-01-q} 
holds if and only if \eqref{SqTest-01-q} and 
\begin{align}
\label{TestAdjPi-q}
\int_{J}\left[\sum_{I\in\mathcal{L}(J)}
\frac{\nu{(I)}}{\mu{(I)}}\mathbb{E}\ci{I}^\nu\left( | b\ci I |^q\right)\right]^{r'}d\mu\le B_*^{r'} \nu(J), \qquad r=p/q, \ 1/r+1/r'=1
\end{align}
hold for all $J\in\cL$. 

Moreover, for the best constants we have $ \max\{B, c(p,q) B_*^{1/q} \}\le A \le C(p,q) \max\{B,B_*^{1/q}\}$. 
\end{theorem}

\subsubsection{Trivial implications and plan of the paper} 
\label{s:TrivImpl} 
Necessity of condition \eqref{SqTest-01-q} and the inequality $B\le A$ in Theorem \ref{t:SqPara-p<q} is easy: one just need to test \eqref{SqEst-01-q} on the function $\1\ci{\!J}$. It is then  easy to see that the left hand side of \eqref{SqTest-01-q} is majorated by $\|\Pi_\beta \1\ci{\!J}\|\ci{\gb_p^q(\mu)}^p$ (the latter includes some extra terms in the sum). 

By the Littlewood--Paley estimates Theorems \ref{thm1} and \ref{t:para-p>2} are equivalent Theorems \ref{t:SqPara-p<2} and \ref{t:SqPara-p>2} respectively for the special choice of $\beta$ given in Remark \ref{r:beta}. By the Littlewood--Paley estimates (Theorem \ref{thm2}) condition \eqref{eq5} is equivalent (up to a constant depending only on $p$) to condition  \eqref{SqTest-01} for this special choice of $\beta$. Condition \eqref{eq6}  is exactly condition \eqref{TestAdjPi} for this special choice of $\beta$. 

Necessity of condition \eqref{TestAdjPi-q} will be explained later in Section \ref{s:Nec*Test}. 

Proof of sufficiency in Theorem \ref{t:SqPara-p<q} is quite easy and will be presented below in Section \ref{s:Suff_p<2} below. 

The rest of the paper is devoted to the proof of sufficiency in Theorem \ref{t:SqPara-p>q}.

\section{The case \texorpdfstring{$1<p\le q$}{1<p =< q}: the sufficiency}
\label{s:Suff_p<2}

Recall that as we explained above in Remark \ref{r:Cu} a slightly more general statement that Theorem \ref{t:SqPara-p<q} was proved in \cite{Cu}. The proof in this section is presented only for the reader's convenience. 

Necessity of  condition \eqref{SqTest-01-q} in Theorem \ref{t:SqPara-p<q} was explained above in Section \ref{s:TrivImpl}. Below we show that this condition is sufficient, i.e.~that it implies estimate \eqref{SqEst-01-q}.

\subsection{Construction of the stopping moments}
\label{s2}

The following standard construction is used for the proof of Theorem \ref{t:SqPara-p<q}, it will also be needed later in the paper.

Given a (finite) collection $\cF\subset\cL$ of the intervals, we construct the collection $\cG\subset \cF$ of the  
stopping moments as follows. Given a non-negative function $f\geq 0$. For $J\in\mathcal{F}$, let $\mathcal{G}^*(J)$ be the collection of maximal intervals $I\in\mathcal{F}$, $I\subseteqq J$ such that 
\[
\langle f\rangle\ci{I,\mu}>2\langle f\rangle\ci{J,\mu}.
\]
Note that intervals from $\mathcal{G}^*(J)$ are pairwisely disjoint. Let $\mathcal{F}(J)=\{I\in\mathcal{F}: I\subseteq J\}$ and let $G(J)=\bigcup_{{I\in\mathcal{G}^*(J)}}I$. Define also $\mathcal{E}(J)=\mathcal{F}(J)\setminus\bigcup_{{I\in\mathcal{G}^*(J)}}\mathcal{F}(I)$. Then we have the following properties
\begin{enumerate}
\item For any $I\in\mathcal{E}(J)$, $\langle f\rangle\ci{I,\mu}\leq 2\langle f\rangle\ci{J,\mu}$,
\item $\mu(G(J))<\frac{1}{2}\mu(J)$.
\end{enumerate}

\par To construct a collection $\mathcal{G}$,  consider all maximal (by inclusion) intervals $J\in\cF $. These sets form the first generation $\mathcal{G}^*_1$ of stopping moments. Inductively define the $(n+1)$-th generation of stopping moments by $\mathcal{G}^*_{n+1}=\bigcup_{{I\in\mathcal{G}^*_n}}\mathcal{G}^*(I)$ and we define the collection of stopping moments by $\mathcal{G}=\cup_{n\geq 1}\mathcal{G}_n^*$.

\par Property (ii) implies that the collection $\mathcal{G}$ of stopping moments satisfies the famous \emph{Carleson measure condition} with constant $2$
\begin{align}\label{CarlEst}
\sum_{I\in\mathcal{G}, I\subseteq J}\mu(I)<2\mu(J), ~~~~J\in\mathcal{L}.
\end{align}
We will use to the following well-know result.

\begin{theorem}[Martingale Carleson Embedding Theorem]
\label{thm3}
Let $\mu$ be a measure on $(\mathcal{X}, \mathcal{T})$ and let $\alpha\ci{I}\geq 0$, $I\in\mathcal{L}$ satisfy the Carleson measure condition 
\begin{align*}
\sum_{I\in\mathcal{G}, I\subseteq J}\alpha\ci{I}\leq C\mu(J).
\end{align*}
Then for any measurable function $f$ and any $1<p<\infty$
\begin{align*}
\sum_{I\in\mathcal{L}}\alpha{\ci{I}}\left |\langle f\rangle\ci{I,\mu}\right |^p\leq C\cdot(p')^p\|f\|\ci{L^p(\mathcal{X}, \mathcal{T}, \mu)}^p\,, 
\end{align*}  
where $1/p'+1/p=1$. 
\end{theorem}

This theorem with some constant instead of $(p')^p$ is a well known fact in harmonic analysis. The constant $(p')^p$ can be obtained, for example, from the fact that the norm of the martingale maximal function in $L^p$ is estimated by $p'$. This connection with the constant in the Carleson Embedding Theorem is explained in  Section 4 of \cite{T1};
as for the constant in estimate of the maximal function see, for example, \cite[Theorem 14.1]{Wil-ProbMart-1991}. 

For a direct proof of Theorem \ref{thm3} see \cite{Lai-CET}; optimality of the constant $(p')^p$ is also proved there. 

\subsection{Sufficiency in Theorem \ref{t:SqPara-p<2}}
First of all notice that it is sufficient to prove \eqref{SqEst-01} only for $f\ge 0$, so let us assume without loss of generality that $f\ge0$.

Using standard approximation reasoning we can also assume without loss of generality that only finitely many of the terms $\beta\ci{I,I'}$ are non-zero. 

Then, applying the construction from the previous section with 
\[
\cF=\{I\in\cL: \exists I'\in\ch(I)\text{ s.t. }\beta\ci{I, I'} \ne 0\}
\]
we get the collection $\cG$ of the stopping moments.


For $J\in \cG$ denote
\[
F\ci{\!J} = \left(\sum_{I\in\cE(J)} \left|\La f \Ra\ci{I,\mu} \right|^q\left| b\ci I\right|^q \right)^{1/q}, 
\]
so 
\[
\left(\sum_{I\in\mathcal{L}}\left|\La f \Ra\ci{I,\mu} \right|^q\left| b\ci I\right|^q\right)^{p/q} =
\left(\sum_{J\in\cG}\sum_{I\in\cE(J)}\left|\La f \Ra\ci{I,\mu} \right|^q\left| b\ci I\right|^q\right)^{p/q}
= \left(\sum_{J\in\cG} |F\ci{\!J}|^q \right)^{p/q}
\]
Note that 
\begin{align}
\label{est-F_J}
\int |F\ci{\!J}|^p d\nu & = \int_\cX \left( \sum_{I\in\cE(J)} \left|\La f \Ra\ci{I,\mu} \right|^q\left| b\ci I\right|^q \right)^{p/q} d\nu 
\\  \notag
& \le 2^p |\La f\Ra\ci{\!J, \mu} |^p \int_\cX \left( \sum_{I\in\cE(J)} \left| b\ci I\right|^q \right)^{p/q}
&& \text{by \cond1  }
\\ \notag
&\le 2^p |\La f\Ra\ci{\!J, \mu} |^p B^p \mu(J)  
&& \text{by \eqref{SqTest-01-q}  }. 
\end{align}
Using the fact that $\|x\|\ci{\ell^q} \le \|x\|\ci{\ell^p}$ for $p\le q$, we can estimate
\begin{align*}
\int_{\mathcal{X}}\left(\sum_{I\in\mathcal{L}}\left|\La f \Ra\ci{I,\mu} \right|^q\left| b\ci I\right|^q\right)^{p/q}d\nu
&  = \int_{\mathcal{X}}\left(\sum_{J\in\mathcal{G}} |F\ci{\!J}|^q \right)^{p/q} d\nu   \\
& \leq \int_{\mathcal{X}}\left(\sum_{J\in\mathcal{G}} |F\ci{\!J}|^p\right) d\nu  && \text{because } \|x\|\ci{\ell^q} \le \|x\|\ci{\ell^p}
\\
& \le 2^p B^p \sum_{J\in\cG} |\La f\Ra\ci{\!J, \mu} |^p  \mu(J)  && \text{by \eqref{est-F_J}  }
\\
& \leq 2^{p+1}(p')^p B^p\|f\|\ci{L^p(\mu)}^p  &&\text{by Theorem \ref{thm3} and \eqref{CarlEst}}.
\end{align*}
\  \hfill \qed

\section{The case \texorpdfstring{$2<p<\infty$}{2<p<infty}: a counterexample}
\label{s:cEx}
In this section we show  that for $2<p<\infty$ condition \eqref{eq5} alone  is not sufficient for \eqref{eq4}.  More precisely, we construct a counterexample, showing that Theorem \ref{t:SqPara-p<2} with the special choice of $\beta$ given in Remark \ref{r:beta} fails  for any $p>2$, i.e.~that condition \eqref{SqTest-01} does not imply \eqref{SqEst-01} in this case. By the Littlewood--Paley estimates (Theorem \ref{thm2}) this will imply that  condition \eqref{eq5}   is not sufficient for \eqref{eq4} for any $p>2$. 

 Consider the real line $\mathbb{R}$ with the Borel $\sigma$-algebra $\mathcal{B}(\mathbb{R})$. Let the $\sigma$-algebra $\fS_n$ be generated by the triadic intervals $I_k^n=3^{-n}[k, k+1)$, $k\in\Z$, so $\cL$ is the collection of all triadic intervals $I_k^n$, $k, n\in\Z$. 
 
 We specify the measures $\mu, \nu$ and the functions $b, f$ in the following way. 

Let $C_n$, $n\ge 0$ be the sets similar%
\footnote{The $\frac13$-Cantor set is defined as $\bigcap_{n\ge0} \overline C_n$, where $\overline C_n$ is the closure of $C_n$.  }
to the ones appearing in the construction of the 
 $\frac{1}{3}$-Cantor set, namely $C_0=[0,1), C_1=[0, 1/3)\cup[2/{3}, 1)$ and, in general, 
\[
C_n=\bigcup_{x=\sum_{j=1}^n\varepsilon_j3^{-j},\ \varepsilon_j\in\{0, 2\}} [x, x+3^{-n}).
\] 
Note that the above intervals $[x, x+3^{-n})$ are the connected components of $C_n$, and $C_n$ is a union of    $2^n$ such intervals.

The measure $\mu$ is taken to be the Lebsgue measure restricted on $[0, 1)$ and the measure $\nu$ is the Cantor measure, i.e. for each triadic interval $I=I^n_k =3^{-n}[k, k+1)$ of length $3^{-n}$ we have  $\nu(I)=2^{-n}$ if $I\subset C_n$ and $\nu(I)=0$ otherwise. 

For the function $b$, we specify its martingale differences $\Delta\ci{I}^\nu b$. Let $I=I_k^n$ be a triadic interval, $I\subset C_n$, and let $I_\pm$ be its right and left thirds respectively (note that $I_\pm\subset C_{n+1}$). 

Assume that $p>2$ is fixed, and 
define $\Delta\ci I^\nu := (2/3)^{n/p} \left( \1\ci{I_+}-\1\ci{I_-}\right)$.

Now let us construct the function $f$. The set $D_n:=C_{n-1}\setminus C_{n}$, $n\ge 1$ is a union of $2^{n-1}$ triadic intervals of length $3^{-n}$, and the sets $D_n$ are clearly disjoint. Fix $r$, $1/p<r<1/2$ and define 
\[
f:= \sum_{n\in\N} (3/2)^{n/p} n^{-r} \1\ci{D_n}. 
\]

\begin{claim}
The above construction gives a counterexample.
\end{claim}

\begin{proof}
We first show that condition \eqref{SqTest-01} with $b\ci I = \Delta\ci I^\nu$ is satisfied. 

Let $J=I_k^n$ be a triadic interval of length $3^{-n}$. If $J\not\subset C_n$, then $\pi_b^{\mu,nu} \1\ci{\!J} =0$ (equivalently, $\Pi_\beta \1\ci{\!J}=0$), so we only need to check the case $J\subset C_n$. 

Note that by the definition of the measures
\[
\mu(J) = 3^{-n}, \qquad \nu(J) = 2^{-n}. 
\]
We then can estimate
\begin{align*}
\left(\sum_{I\in\cL( J)}|\Delta\ci{I}^\nu b|^2\right)^{p/2}
\leq \left(\sum_{k\geq n}\left(2/3\right)^{2k/p}\right)^{p/2}
\le B^p \left(\frac{2}{3}\right)^n,
\end{align*}
where $B=\left(1-(2/3)^{2/p}\right)^{-1/2}$. 
Hence,
\begin{align*}
\int_J\left(\sum_{I\in\cL( J)}|\Delta\ci{I}^\nu b |^2\right)^{p/2}d\nu
\le B^p \left(\frac{2}{3}\right)^n \nu(J)= B^p\left(\frac{2}{3}\right)^n\cdot \left(\frac{1}{2}\right)^n
= B^p \mu(J),
\end{align*}
so \eqref{SqTest-01} holds. 

Next, we show that condition \eqref{SqEst-01} fails. 
Recall that the sets $D_n$ are disjoint, $\mu(D_n) = 2^{n-1}3^{-n}$ and that $f=(3/2)^{n/p} n^{-r}$ on $D_n$. 
So we
can estimate
\begin{align*}
\|f\|\ci{L^p(\mu)}^p= \int_0^1|f|^pdx
=\sum_{n\geq1}\left(3/2\right)^n n^{-pr} 2^{n-1} 3^{-n}
= 2^{-1}\sum_{n\geq1}n^{-pr}<\infty, 
\end{align*}
because $r>1/p$ and so $rp>1$. 

Let us estimate $\sum_{I\in\cL} |\La f\Ra\ci{I,\mu}|^2 \cdot|\Delta\ci I^\nu b|^2$ on $C_n$. 
Take a triadic interval $I$ of length $3^{-k}$, $I\subset C_k$. Then the middle third of this interval is a subset of $D_{k+1}$, so 
\begin{align*}
|\La f \Ra\ci{I,\mu} |\geq 3^{-1}\left(3/2\right)^{({k+1})/{p}}(k+1)^{-r}.
\end{align*}
We can see from the definition of $\Delta\ci I^\nu b$  that on $C_{k+1}\cap I$ (and so on  $C_n\cap I$ for any $n>k$)
\[
|\Delta\ci I^\nu b| = (2/3)^{k/p}. 
\]
so on $C_n\cap I$ we have
\[
|\La f \Ra\ci{I,\mu} | \cdot |\Delta\ci I^\nu b| \ge 3^{-1} (3/2)^{1/p} (k+1)^{-r}. 
\]

For a connected component $I_n$ of $C_n$ there exist connected components $I_k$ of $C_k$, $k=0, 1, \ldots ,n-1$ such that $I_n\subset I_k$. Therefore, squaring   the above inequality and taking  the sum over $k=0, 1, \ldots, n-1$ we get that on $C_n$. 
\[
\sum_{I\in\cL} |\La f \Ra\ci{I,\mu} |^2  |\Delta\ci I^\nu b|^2 \ge  3^{-2} (3/2)^{2/p} \sum_{k=1}^n k^{-2r}
\]
Since $\nu(C_n) =1$ we get that 
\[
\int_{C_n} \left( \sum_{I\in\cL} |\La f \Ra\ci{I,\mu} |^2 |\Delta\ci I^\nu b|^2 \right)^{p/2} d\nu 
\ge  3^{-p} (3/2) \left( \sum_{k=1}^n k^{-r} \right)^{p/2} . 
\]
The right had side of this inequality tends to $\infty$ as $n\to\infty$, so \eqref{SqEst-01} fails. 
\end{proof}

\section{The case \texorpdfstring{$2<p<\infty$}{2<p<infty}:  reduction to the shifted positive martingale operators and necessity of condition \texorpdfstring{\eqref{TestAdjPi-q}}{ (\ref{TestAdjPi-q}) }
}
\label{s:Nec*Test}
In this section we assume that $q$, $1<q<\infty$ is fixed. 
\subsection{An auxiliary positive martingale operator}
Consider an auxiliary non-linear operator $\Phi= \Phi_\beta=\Phi_\beta^q$, 
\[
\Phi_\beta f = \sum_{I\in\cL} |\La f\Ra\ci{I,\mu} |^q |b\ci I|^q ; 
\]
since $q$ is assumed to be fixed, we will skip index $q$ to simplify the notation. 

Estimate \eqref{SqEst-01-q} is equivalent to the estimate
\begin{align}
\label{estPhi}
\| \Phi_\beta f\|\ci{L^{p/q}(\nu)}^{1/q} \le A \|f\|\ci{L^p(\mu)} \qquad \forall f\in L^p(\mu). 
\end{align}
Consider also an auxiliary linear operator $\wt \Pi=\wt \Pi_\beta$, 
\[
\wt \Pi_\beta f = \sum_{I\in\cL} \La f\Ra\ci{I,\mu}   |b\ci I|^q 
\]
By H\"{o}lder inequality $|\La f\Ra\ci{I,\mu} |^p \le \La |f|^p\Ra\ci{I,\mu} $, so  $\Phi_\beta f \le \wt\Pi_\beta (|f|^q)$ (pointwise estimate). Therefore \eqref{SqEst-01-q}  follows from the estimate
\begin{align}
\label{estWtPi}
\| \wt\Pi_\beta g \|\ci{L^{p/q}(\nu)} \le A^{q} \|g\|\ci{L^{p/q}(\mu)} \qquad \forall g\in L^{p/q}(\mu)
\end{align}
(put $g=|f|^q$ in \eqref{estPhi} and notice that $\|g\|\ci{L^{p/q}(\mu)} = \|f\|\ci{L^p(\mu)}^q$). 

The operator $\wt \Pi$ looks almost like the so-called \emph{positive dyadic operators}: if the functions $b\ci I$ were constants on $I$ (recall that $b\ci I$ are supported on $I$) we would get exactly a positive dyadic operator. But we only have that $b\ci I$ is constant on children of $I$, so we say that the operator $\wt \Pi_\beta$ is a \emph{shifted} positive dyadic operator. 

For positive dyadic operators the Sawyer type testing conditions are necessary and sufficient for the two weight estimates. If the same holds for the shifted	positive martingale operators (it is reasonable to expect, and we will prove it later in Section \ref{s:PDO}), we will get necessary and sufficient conditions for two weight estimate for paraproducts. 

Let us write down the testing conditions for $\wt\Pi_\beta$. Applying operator $\wt\Pi_\beta$ to $\1\ci{\!J}$ and taking only terms corresponding to $I\subset J$ in the sum, we get the following necessary condition for the boundedness of $\wt\Pi_\beta$
\begin{align}
\label{test-02}
\int_J \biggl(\sum_{I\in\cL(J)} |b\ci I|^q \biggr)^{p/q} d\nu \le A^p \|\1\ci{\!J}\|\ci{L^{p/q}(\mu)}^{p/q} = A^p \mu(J)
\end{align}
Note that this is exactly the testing condition \eqref{SqTest-01-q} with $A$ instead of $B$. As we discussed above in Section \ref{s:TrivImpl} this condition is also necessary for the boundedness of the original operator operator $\Pi_\beta$. 

Applying the adjoint of $\wt\Pi_\beta$ to $\1\ci{\!J}$ and again taking in the sum only terms corresponding to $I\subset J$ and  denoting by $r'$ the H\"{o}lder conjugate to $r=p/q$, $1/r+1/r'=1$, we get another condition, necessary for the boundedness of $\wt\Pi_\beta$,
\begin{align}
\label{test*-02}
\int_J \biggl(\sum_{I\in\cL(J)} |b\ci I|^q\nu(I)/\mu(I) \biggr)^{r'} d\mu \le A^{r'} \|\1\ci{\!J}\|\ci{L^{r'}(\nu)}^{r'} = A^{r'} \nu(J) . 
\end{align}
Note that this condition is exactly condition \eqref{TestAdjPi-q} with $A$ instead of $B_*$. This condition is clearly necessary for the boundedness of $\Pi_\beta$. The proposition below implies that it is necessary for the boundedness of $\wt\Pi_\beta$.

\begin{proposition}
\label{p:Pi-WtPi}
The operators $\Pi_\beta:L^p(\mu)\to \gb_p^q$ and $\wt\Pi_\beta : L^r(\mu)\to L^r(\nu)$, $r=p/q$ are bounded simultaneously. Moreover, 
\begin{align}
\label{Pi-WtPi}
\| \Pi_\beta\|\ci{ L^p(\mu)\to \gb_p^q(\nu) }^q \le \|\wt\Pi_\beta\|\ci{ L^r(\mu)\to L^r(\nu) } \le 
{4p}(p-q)^{-1} \| \Pi_\beta\|\ci{ L^p(\mu)\to \gb_p^q(\nu) }^q\,. 
\end{align}
\end{proposition}

Let us first notice that the fact that the norms in \eqref{Pi-WtPi} have different powers is not a typo: to see that the powers should be like that one just need to see what happens when we multiply the sequence $\beta$ by a constant $a>0$. 

Note, that the first inequality in \eqref{Pi-WtPi} was already proved, it is just the fact that   \eqref{estWtPi} implies \eqref{estPhi}, so we only need to prove the second estimate. 

\subsection{Rubio de Francia operator and proof of Proposition \ref{p:Pi-WtPi}}

We need the well-known facts about the so-called Rubio de Francia operator. 
Recall that the martingale maximal function $M_\mu= M\ci{\cL, \mu}$ is defined as 
\[
M_\mu f (x) = \sup_{I\in\cL:\, x\in I} |\La f\Ra\ci{I,\mu}| , 
\]
and that for $1<p<\infty$
\[
\| M_\mu f\|\ci{L^p(\mu)} \le p' \| f\|\ci{L^p(\mu)};
\]
here recall $p'$ is the conjugate H\"{o}oder exponent, $1/p+1/p' =1$. 

Let $M_\mu^n$ denotes the iterated maximal function, $M^2_\mu f = M_\mu(M_\mu f)$, $M^{n+1}_\mu f = M_\mu (M_\mu^n f )$. The Rubio de Francia operator $R_p=R _{\mu,p}$, $1<p<\infty$ is defined as 
\[
R_{\mu, p} f = \sum_{n\ge 0} (2p')^{-n} M_\mu^n f; 
\]
here $M_\mu^0 f = |f|$. 

\begin{proposition}
\label{p:RdF}
Let $f\ge 0$, $f\in L^p(\mu)$ Then
\begin{enumerate}
\item $f\le R_{\mu,p} f$;
\item $\| R_{\mu,p} f \|\ci{L^p(\mu)}  \le  2 \|  f \|\ci{L^p(\mu)}  $;
\item For any $I\in\cL$
\[
\inf_{x\in I} R_{\mu,p} f (x) \ge   (2p')^{-1} \La R_{\mu,p} f \Ra\ci{I, \mu}  \,. 
\]
\end{enumerate}
\end{proposition}

The proof is well-known, see for example \cite[Chapter IV, Lemma 5.1]{GC-RDF-WNE-1985}, and pretty  straightforward.  
We leave it as an exercise for the reader. 

\begin{proof}[Proof of Proposition \ref{p:Pi-WtPi}]
Denote $r=p/q$.  For $f\in L^r(\mu)$ define 
\[
\wt f :=  R_{\mu,r} |f|, \qquad g:= \wt f^{1/q}. 
\]
By statement \cond2 of Proposition \ref{p:RdF}
\begin{align}
\label{g^2<f}
\|  g \|\ci{L^p(\mu)}^q = \| \wt f \|\ci{L^r(\mu)}   \le 2 \|  f \|\ci{L^r(\mu)} 
 .
\end{align}
By statement \cond3 of Proposition \ref{p:RdF} for any $I\in\cL$
\begin{align}
\label{RevHold}
\La g^q\Ra\ci{I,\mu} = \La \wt f \Ra\ci{I, \mu}  \le   2r'   \inf_{x\in I} \wt f(x)  = 2r'   \inf_{x\in I} g(x)^q \le  2r' \La  g\Ra\ci{I,\mu}^q ;
\end{align}
the last inequality is the trivial fact that infimum is bounded by the average. 

By statement \cond1 of Proposition \ref{p:RdF} we have the a.e.~estimate
\begin{align}
\label{incrWtPi}
|\wt\Pi_\beta f |\le \wt\Pi_\beta |f| \le \wt\Pi_\beta \wt f .  
\end{align}
By the reverse H\"{o}lder inequality \eqref{RevHold}
\begin{align}
\label{Pif<Pig^2}
\wt\Pi_\beta \wt f \le 2r' \sum_{I\in\cL} |\La g \Ra\ci{I,\mu} |^q |b\ci I|^q  , 
\end{align}
so, taking into account \eqref{incrWtPi} we get
\begin{align*}
\| \wt\Pi_\beta  f \|\ci{L^r(\nu)}  & \le \| \wt\Pi_\beta \wt f \|\ci{L^r(\nu)} && \text{by \eqref{incrWtPi}} \\
& \le 
2r' \|\Pi_\beta g\|\ci{\gb_p^q(\nu)}^q  && \text{by \eqref{Pif<Pig^2}}\\
& \le 
2r'\| \Pi_\beta\|\ci{ L^p(\mu)\to \gb_p^q(\nu) }^q 
\| g\|\ci{\gb_p^q(\nu)}^q \\
& \le 
4r'\| \Pi_\beta\|\ci{ L^p(\mu)\to \gb_p^q(\nu) }^q  \|f\|\ci{L^r(\mu)} && \text{by \eqref{g^2<f} }   ,
\end{align*}
so the second inequality in \eqref{Pi-WtPi} holds. 

And as we discussed immediately after Proposition \ref{p:Pi-WtPi}, the first inequality in \eqref{Pi-WtPi} is already proved. 
\end{proof}

\begin{rem*}

Note that for $p\to q^+$ the constants blow up. Namely, from Proposition \ref{p:Pi-WtPi} we get that the constant $B_*$ in the dual testing condition \eqref{TestAdjPi-q} can be estimated as $B_* \le c(p,q)^{-1} A^q= 4r' A^q$, where $A$ is the estimate of the norm of $\Pi_\beta$. Note that  $c(p,q) \to \infty$ as $p\to q^+$, which means that using our approach we cannot prove the necessity of the dual testing condition \eqref{TestAdjPi-q} in the limit case $p=q$. And in fact,  the limiting case $p=q$ of the condition \eqref{TestAdjPi-q} is not necessary: this can be seen, for example, in the simplest case when $\cL$ is the standard dyadic lattice $\cD$ on $\R$, $p=q=2$ and $\mu=\nu$ is the Lebesgue measure. 
\end{rem*}

\subsection{Bilinear operators and trilinear forms}
The discussion in this subsection is not necessary of the proof of the main result, but it might be of independent interest.

Estimate \eqref{estPhi} for $q=2$  can be interpreted as as an estimate of the bilinear operator $L_\beta : L^p(\mu) \times L^p(\mu) \to L^{p/2}(\nu)$, 
\begin{align}
\label{L_beta}
L_\beta(f,g) = \sum_{I\in\cL} \La f\Ra\ci{ I,\mu} \La g\Ra\ci{ I,\mu} |b\ci I|^2 .
\end{align}

Indeed, by the definition \eqref{L_beta} of $L_\beta$
\[
| L_\beta(f,g) | \le \frac12\Bigl( \Phi_\beta (f) + \Phi_\beta(g) \Bigr),
\]
so if \eqref{estPhi} is satisfied then 
\[
| L_\beta(f,g) | \le A^2 \frac12 \Bigl( \|f\|\ci{L^p(\mu)}^2 + \|g\|\ci{L^p(\mu)}^2 \Bigr). 
\]
Replacing in this inequality $f$ by $tf$ and $g$ by $t^{-1} g$ (note that this does not change the left hand side) and taking the infimum over all $t>0$ in the right hand side we get that 
\begin{align}
\label{estL}
| L_\beta(f,g) | \le A^2  \|f\|\ci{L^p(\mu)}  \|g\|\ci{L^p(\mu)}. 
\end{align}

On the other hand, if \eqref{estL} holds, then plugging in $|f|$ for $f$ and $g$ we get exactly \eqref{estPhi}. 

\subsection{Reduction to the estimates of shifted positive martingale operators} 
\label{s:RedToWtPi}

We reduced Theorem \ref{t:SqPara-p>q} to the estimates of the auxiliary shifted positive martingale operator $\wt \Pi_\beta$. 

Let us write this operator in more symmetric form. Recall that for $I\in\cL$ the symbol $\widehat I$ denotes the parent of $I$. 

For a  sequence $\{\alpha\ci{\widehat I, I}\}\ci{I\in\cL}$, $\alpha\ci{\widehat I, I}\ge 0$ define the operator $T_\alpha$, 
\begin{align}
\label{T_alpha}
T_\alpha f = \sum_{I\in\cL}  \left(\int_Ifd\mu\right) a\ci{I}, \qquad \text{where } a\ci{I}= \sum_{I'\in\ch(I)} a\ci{I, I'}\1\ci{I'}
\end{align}
Note that if $\alpha\ci{ I, I'} = \mu (I)^{-1} |\beta\ci{I,I'}|^p$, so $a\ci I = \mu (I)^{-1} |b\ci I|^p$, then $T_\alpha$ is exactly the operator $\wt\Pi_\beta$. 

Theorem \ref{t:SqPara-p>2} follows from the theorem below, applied with $r=p/q$ instead of $p$.

\begin{theorem}[Two weight estimates for shifted positive operator]\label{thm5}
Let $\alpha=\{\alpha\ci{\widehat I, I}\}\ci{I\in\cL}$ be a sequence of non-negative constants and let $T_\alpha$ be the operator defined by \eqref{T_alpha}

Then 
\[
\|T_\alpha  f \|\ci{L^p(\nu)} \le A \|f\|\ci{L^p(\mu)} \qquad\forall f\in L^p(\mu)
\] 
if and only if 
\begin{align}
\label{eq16}
\int_{J}\biggl(\sum_{I\in\mathcal{L}(J) }\mu(I) a\ci I\biggr)^{p}d\nu
&\le B^p \mu(J), 
\intertext{and}
\label{eq17}
\int_{J}\biggl(\sum_{I\in\mathcal{L}(J)} \nu(I) a\ci I\biggr)^{p'}d\mu
&\le B_*^{p'}\nu(J).
\end{align}
Moreover, for the best constants we have $\max\{B, B_*\} \le A \le C(p) \max\{B, B_*\}$
\end{theorem}

\begin{rem*}
We can see from Proposition \ref{Pi-WtPi} that constants in the above Theorem \ref{thm5} (with $r=p/q$ instead of $p$) are $q$th powers of the constants in Theorem \ref{t:SqPara-p>q}. Note also that constant $A$ in \eqref{test-02} is the same as in Theorem \ref{t:SqPara-p>q}. 
\end{rem*}

\section{Two weight estimates for shifted positive martingale operators}
\label{s:PDO}

In this section we prove Theorem \ref{thm5},  and hence  Theorem \ref{t:SqPara-p>q}. 

Necessity of the conditions \eqref{eq16} and \eqref{eq17} and the estimates $B, B_*\le A$ are trivial: we just test operators $T_\alpha$ and its formal adjoint on functions $\1\ci{\!J}$ and count only part of the sum.  This necessity was already discussed in the previous section, see \eqref{test-02}, \eqref{test*-02}. 

So we only need to to prove the sufficiency. Rescaling $\alpha$  we can assume without loss of generality that $B=B_*=1$ (the constants are not assumed to be optimal).

The idea of the proof comes from  \cite{T1}, but requires some new ideas. 
Namely,  positive martingale operators treated in \cite{T1} are the special case of our operator $T_\alpha$ where all $a\ci I$ are constant on $I$, $a\ci I = \alpha\ci I\1\ci I$, and this property was essential in the proof in \cite{T1}. 

In this paper we only know that $a\ci I$ are constant on children of $I$, so the proof requires modification of the standard stopping moment construction.

Take  $f\geq0$ and $g\geq0$,  $\|f\|\ci{L^p(\mu)} = \|g\|\ci{L^{p'}(\nu)}= 1$. To prove sufficiency in Theorem \ref{thm5} it is enough to show that 
\begin{align}
\label{<Tf,g>}
\La T_\alpha f , g\Ra_\nu  = \sum_{I\in\cL} \alpha\ci{\widehat I, I}\left(\int_{\widehat I}fd\mu\right)\left(\int_{I}gd\nu\right)\le A   .
\end{align}

We split the above sum into two parts by splitting 
$\cL =\cA\cup\cB$ according to the following splitting condition: 
\begin{align}\label{eq18}
\mathcal{A} = \left\{ I\in\mathcal{L}: \langle f \rangle\ci{I,\mu}^p\cdot\mu(I)\geq
\langle g \rangle\ci{I,\nu}^{p'}\cdot\nu(I) \right\}~~\textup{and}~~ \mathcal{B}=\mathcal{L}\setminus\mathcal{A}.
\end{align}

Standard approximation reasoning allows us to assume without loss of generality that only finitely many terms $\alpha\ci{\widehat I,I}$ are non-zero, so all the sums are finite. 


Using the splitting condition \eqref{eq18}, we can write $\La T_\alpha f, g\Ra_\nu =T_1+T_2$, where

\begin{align}\label{eq19}
T_1 = 
\sum_{I\in \cA}\alpha\ci{\widehat I,I}\left(\int_{\widehat I} fd\mu\right)\left(\int_{I}gd\nu\right),
\\
\label{eq20}
T_2 = 
\sum_{I\in\cB}\alpha\ci{\widehat I, I}\left(\int_{\widehat I}fd\mu\right)\left(\int_{I}gd\nu\right).
\end{align}

\subsection{A modified stopping interval construction}
\label{s:ModStop}
To estimate $T_1$ we need to modify a bit the construction of stopping intervals from Section \ref{s2}. The main feature of the construction is that the stopping intervals well be the intervals $I\in \cA$, but the stopping criterion will be checked on their parents $\widehat I$. 

We start with some interval $J$ (not necessarily in $\cA$). For the interval $J$ we define the primary\emph{ preliminary stopping intervals} to be the maximal by inclusion intervals $\widehat I\subset J$, $I\in\cA$, such that 
\begin{align}
\label{StopCond-02}
\La f \Ra\ci{\widehat I, \mu} \ge 2 \La f \Ra\ci{J,\mu}; 
\end{align}
note that different $I\in \cA$ can give the same $\widehat I$, but this  $\widehat I$ is counted only once.

It is obvious that these preliminary stopping intervals are disjoint and their total $\mu$-measure is at most $\mu(J)/2$. 

For each such preliminary stopping interval 
 pick all its children $L$ that belong to $\cA$ (there is at least one such $L$), and declare these children to be the stopping intervals. 

For the children $K\notin \cA$  we continue the process: we will find the maximal by inclusion intervals $\widehat I\subset K$, $I\in \cA$ satisfying \eqref{StopCond-02}, and declare these $\widehat I$ to be the secondary preliminary stopping intervals (note that in the stopping criterion \eqref{StopCond-02} we still compare with the average over the original interval $J$). 

For these preliminary stopping intervals we add their chidden $L\in \cA$ to the stopping intervals, and for the children $K\notin \cA$ we continue the precess (still comparing the averages with the average over the original interval $J$). 

We assumed that the collection $\cA$ is finite, so at some point the process will stop (no $I\in \cA$, $\widehat I\subset K$). We end up with the disjoint collection $\cG^*(J)=\cG\ci\cA^*(J)$ of stopping intervals. 

Since all the stopping intervals are inside the \emph{primary} preliminary stopping intervals, we can conclude that 
\begin{align}
\label{g-decay-02}
\sum_{I\in \cG^*(J)} \mu(I) \le \frac12 \mu(J). 
\end{align}
Denote  $G(J):= \bigcup_{I\in\cG^*(J)} I$. Denoting 
\begin{align*}
\cA(J):=\{I\in\cA: I\subset J\}, \qquad & \cA'(J):=\cA(J) \setminus\{J\}, 
\intertext{define}  
\cE(J) = \cE\ci\cA(J) := \cA(J)\setminus\bigcup_{I\in\cG^*(J)} \cA(I),\qquad & \cE'(J):= \cE(J)\setminus \{J\}.
\end{align*}
It easily follows from the construction  that for any $I\in \cE'(J)$
\begin{align}
\label{aver-est-01}
\La f \Ra\ci{\widehat I,\mu} \le 2 \La f\Ra\ci{J,\mu}. 
\end{align}

\subsection{Estimate of \texorpdfstring{$T_1$}{T<sub>1}}
\label{s:est-T1}
To estimate $T_1$ we run the stopping moments construction defined above in Section \ref{s:ModStop}.
We start with the collection $\cG_0$ of disjoint intervals covering the set $\bigcup_{I\in \cA}\widehat I$. Note that then $\cG_0\cap \cA=\varnothing$. 

For each $I\in \cG_0$ we run the stopping moments construction to get the collection $\cG^*(I)$; the union $\bigcup_{I\in\cG_0}\cG^*(I)$ give us the first generation of stopping moments $\cG_1^*$. Define inductively 
\[
\cG_{k+1}^* := \bigcup_{I\in\cG^*_k} \cG^*(I), 
\]  
and put $\cG:= \bigcup_{k\ge 1} \cG^*_k$. 

Note that the condition \eqref{g-decay-02} implies that the collection $\cG$ satisfies the following Carleson measure condition 
\begin{align}
\label{Carl-02}
|I_0|^{-1} \sum_{I\in \cG: I\subset I_0} \mu(I) \le 2 \qquad \forall I_0\in \cL; 
\end{align}
we also can replace $\cG$ by $\cG\cup\cG_0$ here, and still have the same estimate.

Since the collection $\cA$ is the disjoint union of the collections $\cE'(I)$, $I\in \cG\cup\cG_0$ and the collection $\cG$,  we can 
represent $T_1$ as $T_1=S_1+S_2$, where 

\[
S_1  =\sum_{J\in\cG\cup\cG_0} \int_J F\ci{\!J} g \,d\nu, 
\]
with
\begin{align}
\label{F_J}
F\ci J := \sum_{I\in\cE'(J)} \left(\int_{\widehat I} f d\mu\right) a\ci{\widehat I,I} \bl\ci I , 
\end{align}
and
\begin{align}
\label{S_2}
S_2 = \sum_{I\in\cG} a\ci{\widehat I, I} \left(\int_{\widehat I} f\,d\mu\right) \left(\int_{I} g\,d\nu\right). 
\end{align}
Note, that as defined, $F\ci{\!J}$ is supported by $J$.

Let us estimate $S_1$. The estimate \eqref{aver-est-01} of the average together with the testing condition \eqref{eq16} imply that 
\begin{align}
\label{norm-F_J} 
\|F\ci J\|\ci{L^p(\nu)} \le 2 \La f \Ra\ci{\!J, \mu}\mu(J)^{1/p}
\end{align}

Let us write 
\begin{align*}
S_1 &=\sum_{J\in\cG\cup\cG_0} \int_J F\ci{\!J} g \,d\nu 
\\& = \sum_{J\in\cG\cup\cG_0} \int_{J\setminus G(J)} F\ci{\!J} g \,d\nu +
\sum_{J\in\cG\cup\cG_0} \int_{G(J)} F\ci{\!J} g \,d\nu
\\& =: \sum_{J\in\cG\cup\cG_0} A(J) +  \sum_{J\in\cG\cup\cG_0} B(J) .
\end{align*}
The first sum is easy to estimate. By H\"{o}lder inequality and \eqref{norm-F_J}
\begin{align*}
|A(J)| & \le \|F\ci{\!J}\|\ci{L^p(\nu)} \left( \int_{J\setminus G(J)} |g|^{p'} d\nu\right)^{1/p'}
\\ & \le 2 \La f \Ra\ci{\!J, \mu}\mu(J)^{1/p} \left( \int_{J\setminus G(J)} |g|^{p'} d\nu\right)^{1/p'}
\end{align*}
Then using again H\"{o}lder inequality and then the fact that the sets $J\setminus G(J)$, $J\in \cG\cup \cG_0$ are disjoint we get 
\begin{align*}
\sum_{J\in \cG\cup \cG_0} A(J) &\le 2\left( \sum_{J\in \cG\cup \cG_0}  \La f\Ra\ci{\!J,\mu}^p \mu(J) \right)^{1/p} 
\left( \sum_{J\in \cG\cup \cG_0} \int_{J\setminus G(J)} |g|^{p'} d\nu\right)^{1/p'} \qquad \text{by H\"{o}lder},
\\
& \le 2\left( \sum_{J\in \cG\cup \cG_0}  \La f\Ra\ci{\!J,\mu}^p \mu(J) \right)^{1/p} \|g\|\ci{L^{p'}(\nu)} \qquad \text{sets } J\setminus G(J) \text{ are disjoint.}
\end{align*}

Since the collection $\cG\cup\cG_0$ satisfies the Carleson measure condition \eqref{Carl-02}, we can apply the Carleson Embedding Theorem (Theorem \ref{thm3}) to get 
\begin{align}
\label{f-embed}
\sum_{J\in \cG\cup \cG_0}  \La f\Ra\ci{\!J,\mu}^p \mu(J) \le 2 (p')^p  \|f\|\ci{L^p(\mu)}^p,  
\end{align}
so
\[
\sum_{J\in \cG\cup \cG_0} A(J) \le 2^{1+1/p} p' \|f\|\ci{L^p(\mu)}  \|g\|\ci{L^{p'}(\nu)}. 
\]

To estimate $\sum B(J)$ we notice, that since the functions $F\ci{\!J}$ are constant on each interval $I\in \cG^*(J)$, 
\[
\int_{G(J)} F\ci{\!J} g \,d\nu = \int_{G(J)} F\ci{\!J} g\ci{\!J} \,d\nu, 
\]
where
\[
g\ci{\!J} := \sum_{I\in\cG^*(J)} \La g \Ra\ci{I,\nu} \bl\ci I
\]
Then by H\"{o}lder inequality 
\begin{align*}
B(J) &\le \|F\ci{\!J}\|\ci{L^p(\nu)} \| g\ci{\!J} \|\ci{L^{p'}(\nu)} 
\\ &
\le 2 \La f \Ra\ci{\!J, \mu}\mu(J)^{1/p} \left(\sum_{I\in\cG^*(J)} \La g\Ra\ci{I, \nu}^{p'} \nu(I) \right)^{1/p'}.
\end{align*}
The stopping intervals $I\in\cG^*(J)$ belong to $\cA$, so by \eqref{eq18}
\[
\La g\Ra\ci{I, \nu}^{p'} \nu(I) \le \La f \Ra\ci{I, \mu}^p \mu(I),  
\]
and therefore 
\begin{align*}
B(J) \le 2 \La f \Ra\ci{\!J, \mu}\mu(J)^{1/p} \left(\sum_{I\in\cG^*(J)} \La f \Ra\ci{I, \mu}^p \mu(I) \right)^{1/p'}. 
\end{align*}
Summing and applying H\"{o}oder inequality we get 
\begin{align*}
\sum_{J\in \cG\cup\cG_0} B(J) &\le 2 \left( \sum_{J\in \cG\cup \cG_0}  \La f\Ra\ci{\!J,\mu}^p \mu(J) \right)^{1/p} \left( \sum_{J\in \cG}  \La f\Ra\ci{I,\mu}^p \mu(I) \right)^{1/p'}
\\  &
\le 2  \sum_{J\in \cG\cup \cG_0}  \La f\Ra\ci{\!J,\mu}^p \mu(J)
\\ & 
\le 4 (p')^p \|f\|\ci{L^{p}(\mu)}^p \qquad \text{by \eqref{f-embed} }
\end{align*}
Recall that we assumed that $\|f\|\ci{L^p(\mu)} = \|g\|\ci{L^{p'}(\nu)} = 1$, so 
\[
\sum_{J\in \cG\cup\cG_0} B(J)  \le 4 (p')^p \|f\|\ci{L^{p}(\mu)} \|g\|\ci{L^{p'}(\nu)},  
\]
and thus 
\[
S_1 \le (4 (p')^p+ 2^{1+1/p}p') \|f\|\ci{L^{p}(\mu)} \|g\|\ci{L^{p'}(\nu)}. 
\]

To estimate $S_2$,  let us write recalling the splitting condition \eqref{eq18}
\begin{align*}
S_2 & = \sum_{I\in\cG} a\ci{\widehat I, I} \La f \Ra\ci{\widehat I, \mu } \mu(\widehat I) \La g \Ra\ci{I,\nu} \nu(I) 
\\
& \le \left(\sum_{I\in\cG}  \La f \Ra\ci{\widehat I, \mu }^p a\ci{\widehat I, I}^p \mu(\widehat I)^p \nu(I) \right)^{1/p} \left(\sum_{I\in\cG}\La g \Ra\ci{I,\nu}^{p'} \nu(I) \right)^{1/p'} \qquad \text{H\"{o}lder inequality}
\\ &\le 
\left(\sum_{I\in\cG}  \La f \Ra\ci{\widehat I, \mu }^p a\ci{\widehat I, I}^p \mu(\widehat I)^p \nu(I)\right)^{1/p} 
\left(\sum_{I\in\cG}\La f \Ra\ci{I,\mu}^{p} \mu(I) \right)^{1/p'} \qquad \text{by \eqref{eq18}}
\end{align*}
The second term is easy to estimate: the Carleson measure condition \eqref{Carl-02} and the Carleson Embedding Theorem (Theorem \ref{thm3}) imply that 
\[
\sum_{I\in\cG}\La f \Ra\ci{I,\mu}^{p} \mu(I) \le 2 (p')^p \|f\|\ci{L^p(\mu)}^p. 
\]

To estimate the first term, we apply Lemma \ref{l:CarlCond} below.
\begin{lemma}
\label{l:CarlCond}
The sequence $\{\alpha\ci I\}\ci{I\in\cL}$, 
\[
\alpha\ci I := \sum_{I'\in\child(I)}  a\ci{I, I'}^p \mu(I)^p \nu(I')
\]
satisfies the Carleson measure condition 
\begin{align}
\label{mu-carl-03}
\sum_{I\in \cL: I\subset I_0} \alpha\ci I \le \mu(I_0) \qquad \forall I_0\in\cL. 
\end{align}
\end{lemma}

Using the lemma we can estimate 
\begin{align}
\label{est-S2}
\sum_{I\in\cG}  \La f \Ra\ci{\widehat I, \mu }^p a\ci{\widehat I, I}^p \mu(\widehat I)^p \nu(I) & \le 
\sum_{I\in \cL} \La f \Ra\ci I^p \alpha\ci I &&\text{sum over a bigger set}
\\ 
\notag
& \le 
(p')^p \|f\|\ci{L^p(\mu)}^p &&\text{by Theorem \ref{thm3}}. 
\end{align}
Gathering the above estimates, we get the desired estimate of $S_2$ and therefore of $T_1$. \hfill\qed

\begin{proof}[Proof of Lemma \ref{l:CarlCond}]
For $I_0\in\cL$
\begin{align*}
\sum_{I\in\cL(I_0)} \alpha\ci I & = \int_{I_0} \left( \sum_{I\in\cL(I_0)}\sum_{I'\in\child(I)} 
a\ci{I, I'}^p \mu(I)^p \bl\ci{I'}\right) \, d\nu \\
& \le 
\int_{I_0} \left( \sum_{I\in\cL(I_0)}\sum_{I'\in\child(I)} 
a\ci{I, I'} \mu(I) \bl\ci{I'}\right)^p \,d\nu &&\text{because $\|x\|\ci{\ell^p}\le \|x\|\ci{\ell^1}$}
\\ &
\le \mu(I_0) && \text{by assumption \eqref{eq16}}.
\end{align*}
\end{proof}

\subsection{Estimate of \texorpdfstring{$T_2$}{T<sub>2}} The estimate of $T_2$ is simpler, because it relies on  a simpler stopping moment construction. 
 
Namely,  we run the stopping intervals construction described in Section \ref{s2}  with $\cF=\cB$ and with respect to the measure $\nu$ and the function $g$. 

We start with a collection $\cG_0$ of disjoint intervals covering the set $\bigcup_{I\in\cB} I$, and run the construction starting from these intervals. We get the collection $\cG$ of stopping intervals satisfying the Carleson measure condition
\begin{align}
\label{Carl-nu}
\sum_{I\in \cG\cup \cG_0, \,I\subset I_0} \nu(I) \le 2 \nu(I_0) \qquad \forall I_0\in \cL. 
\end{align}

Again, define $G(J):=\bigcup_{I\in\cG^*(J)}I$. Denoting
\[
\cB(J):=\{I\in \cB:I\subset \cB\}, \qquad \cB'(J):=\{I\in \cB:I\subsetneqq \cB\} =\cB(J)\setminus\{J\}, 
\]
define 
\[
\cE(J)=\cE\ci\cB(J) := \cB(J)\setminus \bigcup_{I\subset \cG^*(J)} \cB(J), \qquad
\cE'(J)=\cE'\ci\cB(J) = \cE(J)\setminus\{J\}. 
\]

Similarly to the case of $T_1$ we split the sum into 2 parts, $T_2= S_1+ S_2$, where $S_2$ is the sum over $\cG$ and $S_1$ is the rest, 
\begin{align*}
T_2= \sum_{I\in \cB:I\notin \cG} a\ci{\widehat I, I} \left(\int_{\widehat I} f \,d\mu\right) \left(\int_I g\,d\nu \right) + \sum_{I\in\cG} a\ci{\widehat I, I} \left(\int_{\widehat I} f \,d\mu\right) \left(\int_I g\,d\nu \right) =: S_1+S_2. 
\end{align*}
Denoting for $J\in \cG\cup\cG_0$  
\[
F\ci{\!J} := \sum_{I\in\cE'(J)} a\ci{\widehat I, I} \left(\int_I g\,d\nu \right) \bl\ci{\widehat I}
\]
and noticing that $F\ci{\!J}$ is supported on $J$ we can write
\[
S_1=\sum_{J\in\cG\cup\cG_0} \int_J F\ci{\!J} f \, d\mu. 
\]
Using the estimate \eqref{Carl-nu} and the testing condition \eqref{eq17} we can write 
\begin{align}
\label{norm-G_J} 
\|F\ci J\|\ci{L^{p'}(\mu)} \le 2 \La g \Ra\ci{\!J, \nu}\nu(J)^{1/p'}
\end{align}

We then  decompose $S_1$ as 
\begin{align*}
S_1 &=\sum_{J\in\cG\cup\cG_0} \int_J F\ci{\!J} f \,d\mu 
\\& = \sum_{J\in\cG\cup\cG_0} \int_{J\setminus G(J)} F\ci{\!J} f \,d\mu +
\sum_{J\in\cG\cup\cG_0} \int_{G(J)} F\ci{\!J} f \,d\mu
\\& =: \sum_{J\in\cG\cup\cG_0} A(J) +  \sum_{J\in\cG\cup\cG_0} B(J) .
\end{align*}
The first sum is easy to estimate. By H\"{o}lder inequality and \eqref{norm-G_J}
\begin{align*}
|A(J)| & \le \|F\ci{\!J}\|\ci{L^{p'}(\mu)} \left( \int_{J\setminus G(J)} |f|^{p} d\mu\right)^{1/p}
\\ & \le 2 \La g \Ra\ci{\!J, \nu}\nu(J)^{1/p'} \left( \int_{J\setminus G(J)} |f|^{p} d\mu\right)^{1/p}
\end{align*}
Then using again H\"{o}lder inequality and then the fact that the sets $J\setminus G(J)$, $J\in \cG\cup \cG_0$ are disjoint we get 
\begin{align*}
\sum_{J\in \cG\cup \cG_0} A(J) &\le 2\left( \sum_{J\in \cG\cup \cG_0}  \La g\Ra\ci{\!J,\nu}^{p'} \nu(J) \right)^{1/p'} 
\left( \sum_{J\in \cG\cup \cG_0} \int_{J\setminus G(J)} |f|^{p} d\mu\right)^{1/p} \qquad \text{by H\"{o}lder},
\\
& \le 2\left( \sum_{J\in \cG\cup \cG_0}  \La g\Ra\ci{\!J,\nu}^{p'} \nu(J) \right)^{1/p'} \|f\|\ci{L^{p}(\mu)} \qquad \text{sets } J\setminus G(J) \text{ are disjoint.}
\end{align*}

Since the collection $\cG\cup\cG_0$ satisfies the Carleson measure condition \eqref{Carl-nu}, we can apply the Carleson Embedding Theorem (Theorem \ref{thm3}) to get 
\begin{align}
\label{g-embed}
\sum_{J\in \cG\cup \cG_0}  \La g\Ra\ci{\!J,\nu}^{p'} \nu(J) \le 2 p^{p'}  \|g\|\ci{L^{p'}(\nu)}^{p'},  
\end{align}
so
\[
\sum_{J\in \cG\cup \cG_0} A(J) \le 2^{1+1/p'} p \|f\|\ci{L^p(\mu)}  \|g\|\ci{L^{p'}(\nu)}. 
\]

To estimate $\sum B(J)$ we notice, that since the functions $F\ci{\!J}$ are constant on each interval $I\in \cG^*(J)$, 
\[
\int_{G(J)} F\ci{\!J} f \,d\mu = \int_{G(J)} F\ci{\!J} f\ci{\!J} \,d\mu, 
\]
where
\[
f\ci{\!J} := \sum_{I\in\cG^*(J)} \La f \Ra\ci{I,\mu} \bl\ci I
\]
Then by H\"{o}lder inequality 
\begin{align*}
B(J) &\le \|F\ci{\!J}\|\ci{L^{p'}(\mu)} \| f\ci{\!J} \|\ci{L^{p}(\mu)} 
\\ &
\le 2 \La g \Ra\ci{\!J, \nu}\nu(J)^{1/p'} \left(\sum_{I\in\cG^*(J)} \La f\Ra\ci{I, \mu}^{p} \mu(I) \right)^{1/p}.
\end{align*}
The stopping intervals $I\in\cG^*(J)$ belong to $\cB$, so by \eqref{eq18}
\[
\La f \Ra\ci{I, \mu}^p \mu(I) < \La g\Ra\ci{I, \nu}^{p'} \nu(I) \,  
\]
and therefore 
\begin{align*}
B(J) \le 2 \La g \Ra\ci{\!J, \nu}\nu(J)^{1/p'} \left(\sum_{I\in\cG^*(J)} \La g \Ra\ci{I, \nu}^{p'} \nu(I) \right)^{1/p}. 
\end{align*}
Summing and applying H\"{o}oder inequality we get 
\begin{align*}
\sum_{J\in \cG\cup\cG_0} B(J) &\le 2 \left( \sum_{J\in \cG\cup \cG_0}  \La g\Ra\ci{\!J,\nu}^{p'} \nu(J) \right)^{1/p'} \left( \sum_{J\in \cG}  \La g\Ra\ci{I,\nu}^{p'} \nu(I) \right)^{1/p}
\\  &
\le 2  \sum_{J\in \cG\cup \cG_0}  \La g\Ra\ci{\!J,\nu}^{p'} \nu(J)
\\ & 
\le 4 p^{p'} \|g\|\ci{L^{p'}(\nu)}^{p'} \qquad \text{by \eqref{g-embed} }
\end{align*}
Recall that we assumed that $\|f\|\ci{L^p(\mu)} = \|g\|\ci{L^{p'}(\nu)} = 1$, so 
\[
\sum_{J\in \cG\cup\cG_0} B(J)  \le 4 p^{p'} \|f\|\ci{L^{p}(\mu)} \|g\|\ci{L^{p'}(\nu)},  
\]
and thus 
\[
S_1 \le (4 p^{p'}+ 2^{1+1/p'}p) \|f\|\ci{L^{p}(\mu)} \|g\|\ci{L^{p'}(\nu)}. 
\]

To estimate $S_2$,  let us write using H\"{o}lder inequality 
\begin{align*}
S_2 & = \sum_{I\in\cG} a\ci{\widehat I, I} \La f \Ra\ci{\widehat I, \mu } \mu(\widehat I) \La g \Ra\ci{I,\nu} \nu(I) 
\\
& \le \left(\sum_{I\in\cG}  \La f \Ra\ci{\widehat I, \mu }^p a\ci{\widehat I, I}^p \mu(\widehat I)^p \nu(I) \right)^{1/p} \left(\sum_{I\in\cG}\La g \Ra\ci{I,\nu}^{p'} \nu(I) \right)^{1/p'} 
\end{align*}
The second term was already estimated in \eqref{g-embed}. The first term is  estimated using Lemma \ref{l:CarlCond}, exactly as it was done in the end of Section \ref{s:est-T1}, see \eqref{est-S2}.

\end{document}